\documentclass[12pt,reqno]{article}

\usepackage[usenames]{color}
\usepackage{amssymb}
\usepackage{amsmath}
\usepackage{amsthm}
\usepackage{amsfonts}
\usepackage{amscd}
\usepackage{graphicx}

\usepackage[colorlinks=true,
linkcolor=webgreen,
filecolor=webbrown,
citecolor=webgreen]{hyperref}

\definecolor{webgreen}{rgb}{0,.5,0}
\definecolor{webbrown}{rgb}{.6,0,0}

\usepackage{color}
\usepackage{fullpage}
\usepackage{float}

\usepackage{graphics}
\usepackage{latexsym}
\usepackage{epsf}
\usepackage{breakurl}
\usepackage{amsmath}

\newcommand{\R}{{\mathbb R}}

\newcommand{\Z}{{\mathbb Z}}
\newcommand{\tarc}{\mbox{\large$\frown$}}
\newcommand{\arc}[1]{\stackrel{\tarc}{#1}}

\begin{document}

\begin{center}
\epsfxsize=4in
\end{center}

\theoremstyle{plain}
\newtheorem{theorem}{Theorem}
\newtheorem{corollary}[theorem]{Corollary}
\newtheorem{lemma}[theorem]{Lemma}
\newtheorem{proposition}[theorem]{Proposition}

\theoremstyle{definition}
\newtheorem{definition}[theorem]{Definition}
\newtheorem{example}[theorem]{Example}
\newtheorem{conjecture}[theorem]{Conjecture}

\theoremstyle{remark}
\newtheorem{remark}[theorem]{Remark}

\begin{center}
\vskip 1cm{\LARGE\bf  Lattice Points Close to a Helix
}

\vskip 1cm
\large
Jack Dalton and Ognian Trifonov\\
University of South Carolina\\
Department of Mathematics \\
Columbia, SC 29208 \\
\href{mailto:jrdalton@math.sc.edu}{\tt jrdalton@math.sc.edu} \\
\href{mailto:trifonov@math.sc.edu}{\tt trifonov@math.sc.edu} \\
\end{center}

\vskip .2 in

\begin{abstract}
We obtain lower bound for the  maximum distance between any three distinct   points in an  affine  lattice which are  close to a helix with small curvature  and  torsion.  
\end{abstract}
 
\section{Introduction}

The problem of estimating the number of lattice points on or close to a curve has a rich history. In  1926,
  Jarn\'{\i}k \cite{Jar}   proved that
  the number of integer
points on a strictly convex closed curve of length $L >3$ does not exceed
$3(2\pi)^{-1/3}L^{2/3} + O \left ( L^{1/3} \right )$ and the exponent and the
constant of the leading term are best possible. Assuming higher order smoothness conditions  on the curve, a number of authors achieved sharper estimates, in particular Swinnerton-Dyer \cite{SwD} and Bombieri and Pila \cite{BomP}.

Estimating the number of lattice points {\it close} to a curve is more recent topic. Of a survey of results on this topic and their applications one may see \cite{FT} and \cite{FGT}.

In 1972  Zygmund published a paper on  spherical summability of Fourier series in two dimensions where an essential component was the following theorem. 

\begin{theorem} (Schinzel)   An arc of length $\sqrt [3]{2} R^{1/3}$ on a circle of radius $R$ contains no more than two lattice points. 
\end{theorem} 

The original proof is due to Schinzel but the following short proof was provided by Pelczynski.

\begin{proof}

Let $A(x_1,y_1)$, $B(x_2,y_2)$, and $C(x_3,y_3)$ be distinct lattice points on a circle of radius $R$. 

Denote by $a$, $b$, and $c$ the lengths of the sides of $\triangle ABC$, and its area by $S$. 
Since the circle is strictly convex curve, we have 
$S = \frac{1}{2} | 
\left | 
\begin{array}{ccc}
x_1 & y_1 & 1\\
x_2 & y_2 & 1\\
x_3 & y_3 & 1
\end{array} \right | | \ \geq {1 \over 2}$. 

On the other hand, by a formula attributed to Heron of Alexandria, we have  ${\displaystyle S = {abc \over 4R}}$. 

Thus, $abc \geq 2R$. 
\end{proof}
 
In the paper \cite{HowTri}  Howard and the second author generalized Schinzel's result to the case of plane curve with bounded curvature and  lattice points close to a plane affine lattice (where plane affine lattice is the set 
$\{ {\bf v}_0 + m{\bf v}_1 + n{\bf v}_2 : m \in {\mathbb Z}, n \in {\mathbb Z} , \}$ with 
${\bf v}_0, {\bf v}_1, {\bf v}_2$, and plane vectors in  such that ${\bf v}_1$ and  ${\bf v}_2$ are linearly independent). 

There are only a few papers on the topic of lattice points close to a three-dimensional curve. 

Huang \cite{Hua} obtains estimates for the number of lattice points close dilations of a curve with parametrization $(x, f_1(x), f_2(x))$, 
$x \in [a,b]$. Huang's estimates depend on the dilation parameter $q$, the upper bound $\delta $ for the distance between the lattice points and the curve, and a constant depending on the functions $f_1$ and $f_2$ which is not explicitly computed.

 In Chapter 4 of his PhD dissertation Letendre  also obtains  estimates for the number of lattice points close dilations of a curve with parametrization $(x, f_1(x), f_2(x))$, 
$x \in [a,b]$. His results only assume  bounds on certain quantities involving derivatives of $f_1$ and $f_2$. 

We will consider the special case when the curve is a helix and obtain lower bound for the maximal distance between any three distinct points in a general lattice close to a helix. 
We define a lattice in three-space as follows.

\begin{definition}
Let $v_0,v_1,v_2,v_3 \in \R^2$ be vectors in ${\mathbb R}^3$ with $v_1$ , $v_2$, and $v_3$
linearly independent.  Then, the \it{affine lattice generated by $v_1$,  $v_2$, and $v_3$  with origin $v_0$} is
$$
\mathcal{L}=\mathcal{L}(v_0,v_1,v_2,v_3) = \{ v_0+m v_1 + nv_2 + pv_3: m,n,p \in \Z\}.
$$
 \end{definition}

We will show in Section 2 that  for each lattice ${\mathcal L}$ there exist positive constants ${\mathcal D_L}$ and ${\mathcal A_L}$  (depending on ${\mathcal L}$) such that if $A_0$ and $A_1$ are two distinct lattice points in 
${\mathcal L}$ then $|A_0A_1| \geq {\mathcal D_L}$; furthermore, if $A_0$, $A_1$, and $A_2$ are three non-collinear lattice points in ${\mathcal L}$, then   Area$(\triangle A_0A_1A_2) \geq {\mathcal A_L}$.

Our main result is

\begin{theorem} \label{main}
Let ${\mathcal H}$ be a helix of curvature $\kappa > 0$ and torsion $\tau > 0$, and let $$0 \leq \delta \leq \min \left (\frac{\mathcal D_L}{4},   \frac{{\mathcal D_L}^2}{11\pi^3}(\kappa^2 + \tau^2),   
\frac{2{\mathcal A_L}}{11\pi }(\kappa^2 + \tau^2)^{\frac{1}{2}} \right ).$$ 
Let ${\mathcal L}$ be affine lattice, and let $A_0$, $A_1$, and $A_2$ be three distinct  points in ${\mathcal L}$  which are within $\delta $ of the helix  ${\mathcal H}$. Then,   
the maximal distance between $A_0$, $A_1$, and $A_2$ is at least $$\min\left ( 1.2{\mathcal A_L}^{\frac{1}{3}} \kappa^{-\frac{1}{3}} - 2\delta,  \frac{\pi \tau}{\kappa^2 + \tau^2} - 2\delta  \right ) .$$
\end{theorem}

The above theorem is a generalization of Schinzel's result to the case of a helix. 

\begin{corollary} \label{cor:1}
Let ${\mathcal H}$ be a helix of curvature $\kappa > 0$ and torsion $\tau > 0$ with $\tau \kappa^{\frac{1}{3}} \geq 0.4(\kappa^2 + \tau ^2)$.  Let $A_0$, $A_1$, and $A_2$ be three distinct lattice points  on ${\mathcal H}$.
Then,   
the maximal distance between $A_0$, $A_1$, and $A_2$ is at least ${\displaystyle 1.1 \kappa^{-\frac{1}{3}}}$.
\end{corollary}

The paper is organized as follows. In Section 2 we prove some auxiliary lemmas. In Section 3 we prove the main result. 

We expect that an analogue of our main result  will hold for curves such that $0 < c_1K < \kappa(s) < c_2K$ and $0 < c_3T < \tau(s) < c_4T$ where $c_1,c_2,c_3$, and $c_4$ are positive constants.

\section{Some auxiliary results}

In this section we prove several lemmas which will be needed in the proof of Theorem \ref{main}.

 \begin{lemma} \label{lem:1}
The area of a triangle in $3$-space with non-collinear lattice point vertices is at least $\frac{1}{2}$.
\end{lemma}

\begin{proof}

Let $A, B$, and $C$ be non-collinear lattice points in $\mathbb{R}^3$. Then, the area of $\triangle ABC$ equals $\frac{1}{2} \left |  \overrightarrow{AB} \times  \overrightarrow{AC} \right |$.
Since $A$, $B$, and $C$ are lattice points, the components of the vectors $\overrightarrow{AB}$, $\overrightarrow{AC}$ 
are integers. Therefore, the components of the cross  product  $\overrightarrow{AB} \times  \overrightarrow{AC}$ are integers, as well. Thus, Area$(\triangle ABC) = \frac{1}{2} \sqrt{n}$ for some integer $n$. Since the points 
$A$, $B$, and $C$ are non-collinear, Area$(\triangle ABC) \neq 0$, so $n \neq 0$. \end{proof}

 The next two lemmas where proved in \cite{HowTri} in the ${\mathbb  R}^2$ case.  The lemmas are easily extended to the $n$-dimensional case. 
 
\begin{lemma} \label{lem:tri1pt}
Let $A$, $B$, $C$, and $C_1$ be points in ${\mathbb  R}^n$ where $n \geq 2$. Let $\delta \geq 0$ and $|CC_1| \leq \delta$. Denote by $S$ the area of  $\triangle ABC$, and by $S_1$ the area of  $\triangle ABC_1$. Then, 
$$|S - S_1| \leq \frac{\delta |AB|}{2}.$$
\end{lemma} 

\begin{proof}
Let $h$ be the distance from the point $C$ to the line $AB$, and let $h_1$ be the distance from the point $C_1$ to the line $AB$.  Then, ${\displaystyle S = \frac{|AB| h}{2}}$ and ${\displaystyle S_1 = \frac{|AB| h_1}{2}}$.

By the triangle inequality, $$h_1 \leq h + |CC_1| \leq h + \delta \quad \mbox{ and  } \quad h \leq h_1 + |CC_1| \leq h_1 + \delta.$$ Therefore, $|h - h_1| \leq \delta$, so ${\displaystyle |S - S_1| = \frac{|AB| |h-h_1|}{2} \leq \frac{\delta |AB|}{2}}$.
\end{proof}

\begin{lemma} \label{lem:4}
Let  $\triangle ABC$  and  $\triangle A_1B_1C_1$ be triangles in ${\mathbb  R}^n$ where $n \geq 2$,  with areas $S$ and $S_3$, respectively.
Let $\delta \geq 0$ and assume
$$|AA_1| \leq \delta,\  |BB_1| \leq \delta, \mbox{ and } |CC_1| \leq \delta.$$

 Then, $$|S-S_3| \leq \frac{(|AB|+|AC|+|BC|)\delta}{2}+\frac{3\delta^2}{2}.$$
\end{lemma}

\begin{proof}
Let $S_1$ be the area of $\triangle ABC_1$, and let $S_2$ be the area of $\triangle AB_1C_1$. Then
\begin{align*}
	|S-S_3| &= |S-S_1 + S_1-S_2+S_2-S_3| \\
	&\leq |S-S_1| + |S_1-S_2|+|S_2-S_3|.
\end{align*}

By Lemma \ref{lem:tri1pt}, ${\displaystyle |S - S_1| \leq \frac{\delta |AB|}{2}}$. Applying Lemma \ref{lem:tri1pt} to the points $A$, $C_1$, $B$, and $B_1$, we get ${\displaystyle |S_1  - S_2| \leq \frac{\delta |AC_1|}{2}}$.
Finally, applying Lemma \ref{lem:tri1pt} to the points $B_1$, $C_1$, $A$, and $A_1$, we get ${\displaystyle |S_2  - S_3| \leq \frac{\delta |B_1C_1|}{2}}$.

Therefore, 
\begin{equation} \label{S-S_3}
|S - S_3| \leq \frac{\delta (|AB| + |AC_1| + |B_1C_1|)}{2}. 
\end{equation}

Note that the the triangle inequality gives us $|AC_1| \leq |AC| + |CC_1| \leq |AC| + \delta$ and $|B_1C_1| \leq |B_1B| + |BC| + |CC_1| \leq |BC| + 2\delta$. 
Using the upper bounds for $|AC_1|$ and $|B_1C_1|$ in \eqref{S-S_3}  gives us the desired bound for $|S - S_3|$.
\end{proof}

We also need the following simple lemma.

\begin{lemma} \label{lem:2}
For all $x$ in the interval $(0, \pi)$, we have $0<\sin x-x\cos x<\frac{x^3}{3}$. Also, for all  $x$ in the interval $(0,  \pi /2]$, we have ${\displaystyle \sin x \geq \frac{2}{\pi } x}$.
\end{lemma}

\begin{proof}
Let $f(x)=\sin x-x\cos x$. Note $f(0)=0$. Consider $f'(x)=x\sin x$. Clearly,  $f'(x)>0$ on $(0, \pi)$ and we have the first half of the inequality.

Now let $g(x)=\sin x-x\cos x-\frac{x^3}{3}$. Note $g'(x)=x(\sin x -x)$. If $x>0$, $\sin x-x<0$ and thus $g'(x)<0$ on this interval also. Thus, because $g(0)=0$ and $g$ is decreasing on $(0, \pi)$, we have the second desired inequality.
To prove the last inequality of the lemma, consider the function ${\displaystyle h(x) = \frac{\sin x}{x}}$. We have ${\displaystyle h'(x) = \frac{x \cos x - \sin x}{x^2}}$. We proved above that $\sin x-x\cos x > 0$ for $x \in (0 , \pi)$. Thus $h(x)$ is decreasing in 
$(0 , \pi /2]$ and its minimum is $h(\pi /2) = 2  / \pi$.
\end{proof}

Here, and throughout the end of the paper, when $A$ is point and $\vec{v}$ is a vector, the equation $A = \vec{v}$ will mean $\overrightarrow{OA} = \vec{v}$, where $O$ is the origin of the coordinate system.

A helix is a space curve with constant curvature and torsion. 
Let $A_0< A_1, A_2$  be three points on a helix ${\mathcal H}$. In the next lemma we obtain formulas for the distances $|A_0A_1$, $|A_1A_2|$, $|A_0A_2|$, and for the area of triangle 
$\triangle A_0A_1A_2$. 

\begin{lemma} \label{sthelix}
Let ${\mathcal H}$ be a helix with curvature $\kappa >0 $ and torsion $\tau > 0$.   
Let $A_0$, $A_1$, and $A_2$ be three distinct points   on the helix ${\mathcal H}$   whose arclengths from the origin of the helix are $s_0$, $s_1$, $s_2$ respectively, with 
$ s_0 < s_1 < s_2$. Define $a = \frac{\kappa}{\kappa^2 + \tau^2}$,    $ b = \frac{\tau}{\kappa^2 + \tau^2}$, $t_i = \frac{s_i}{\sqrt{a^2+b^2}}$ for $i=0,1,2$, $h_1=t_1-t_0$,  and $h_2=t_2-t_1$. Then,

(a) $|A_0A_1|^2=4a^2\sin^2\left(\frac{h_1}{2}\right)+b^2h_1^2$, $|A_1A_2|^2=4a^2\sin^2\left(\frac{h_2}{2}\right)+b^2h_2^2$, and
$|A_0A_2|^2=4a^2\sin^2\left(\frac{h_1+h_2}{2}\right)+b^2(h_1+h_2)^2$;

(b) $(\text{Area}(\triangle A_0 A_1 A_2))^2 =(T_1 + T_2+T_3)/4$, where 

$T_1 :=16a^4\sin^2\frac{h_1}{2}\sin^2\frac{h_2}{2}\sin^2\frac{h_1+h_2}{2}$, 

$T_2:= a^2b^2h_1^2h_2^2\left(\frac{\sin \frac{h_2}{2}}{\frac{h_2}{2}} -\frac{\sin \frac{h_1}{2}}{\frac{h_1}{2}}\right)^2$, and

$ T_3 :=  16a^2b^2h_1h_2\sin \frac{h_1}{2} \sin \frac{h_2}{2}\sin^2\frac{h_1+h_2}{4}$. 
\end{lemma}

\begin{proof}
First, we consider a helix ${\mathcal H}(a,b)$ with parametrization 
\begin{equation} \label{sth}
\vec{r}(t)=\langle a\cos t, a \sin t, bt \rangle. 
\end{equation}
It is easy to check that we have ${\displaystyle  \kappa = \frac{a}{a^2+b^2}}$,  ${\displaystyle \tau = \frac{b}{a^2+b^2}}$, and 
 ${\displaystyle \kappa^2+\tau^2=\frac{1}{a^2+b^2}}$. From here, we get ${\displaystyle a = \frac{\kappa}{\kappa^2 + \tau^2}}$,    ${\displaystyle  b = \frac{\tau}{\kappa^2 + \tau^2}}$. Also, for the natural parameter $s$ we have
 $s = t\sqrt{a^2+b^2}$.

We have  $$A_0=(a\cos t_0, a\sin t_0, bt_0),  A_1=(a\cos t_1, a\sin t_1, bt_1), \mbox{ and  } A_2=(a\cos t_2, a\sin t_2, bt_2),$$ with $t_0<t_1< t_2$.

Therefore, 
\begin{align*}
	\overrightarrow{A_1A_0} &= \langle a(\cos t_0 - \cos t_1), a(\sin t_0 - \sin t_1), b(t_0 - t_1)\rangle, \text{ and } \\
	\overrightarrow{A_1A_2} &= \langle a(\cos t_2 - \cos t_1), a(\sin t_2 - \sin t_1), b(t_2 - t_1)\rangle.
\end{align*}
Therefore, 

\begin{align*}
	\left|\overrightarrow{A_1A_0}\right|^2 &=a^2(\cos t_0-\cos t_1)^2+a^2(\sin t_0-\sin t_1)^2+b^2(t_0-t_1)^2\\
	&=a^2(2 -2\cos t_0 \cos t_1-2\sin t_0 \sin t_1)+ b^2(t_0-t_1)^2 \\
	&=a^2(2 -2\cos (t_0 -t_1))+ b^2(t_0-t_1)^2 \\
	&=a^2(2 -2\cos h_1)+ b^2h_1^2 \\
	&=4a^2\sin^2\left(\frac{h_1}{2}\right)+b^2h_1^2.
\end{align*}
One can obtain the formulas for $|A_1A_2|$ and $|A_0A_2|$ similarly.

For the area of $\triangle A_0A_1A_2$ we use 
\begin{align*}
	\cos\theta&=\frac{\overrightarrow{A_1A_0}\cdot \overrightarrow{A_1A_2}}{\left|\overrightarrow{A_1A_0}\right| \left|\overrightarrow{A_1A_2}\right|} \\
	1-\sin^2 \theta &=\frac{\left(\overrightarrow{A_1A_0}\cdot \overrightarrow{A_1A_2}\right)^2}{\left|\overrightarrow{A_1A_0}\right|^2 \left|\overrightarrow{A_1A_2}\right|^2} \\
	\left|\overrightarrow{A_1A_0}\right|^2 \left|\overrightarrow{A_1A_2}\right|^2-\left(\overrightarrow{A_1A_0}\cdot \overrightarrow{A_1A_2}\right)^2&=\left|\overrightarrow{A_1A_0}\right|^2 \left|\overrightarrow{A_1A_2}\right|^2\sin^2 \theta   \\
	&=4\text{Area}(\triangle A_0 A_1 A_2)^2,
\end{align*}
where $\theta $ is the angle between $\overrightarrow{A_1A_0}$ and $\overrightarrow{A_1A_2}$. Next we calculate the dot product:

\begin{align*}	
	\overrightarrow{A_1A_0}\cdot\overrightarrow{A_1A_2} &= a^2(\cos t_0 - \cos t_1)(\cos t_2 - \cos t_1)\\
	&+ a^2(\sin t_0 - \sin t_1)(\sin t_2 - \sin t_1)+b^2(t_0 - t_1)(t_2 - t_1) \\
	&=a^2(\cos t_0 \cos t_2- \cos t_1\cos t_2 -\cos t_0 \cos t_1+ \cos^2 t_1) \\
	&+ a^2(\sin t_0 \sin t_2 - \sin t_1 \sin t_2 - \sin t_0 \sin t_1 + \sin^2 t_1)+b^2(t_0 - t_1)(t_2 - t_1)  \\
	&=a^2(\cos (h_2+h_1) - \cos h_2 -\cos h_1 + 1)-b^2h_1h_2 \\
	&=a^2((1-\cos h_2)(1-\cos h_1) -\sin h_2 \sin h_1)-b^2h_1h_2 \\
	&=4a^2\sin \frac{h_1}{2} \sin \frac{h_2}{2}\left( \sin \frac{h_1}{2} \sin \frac{h_2}{2}- \cos\frac{h_1}{2}\cos\frac{h_2}{2}\right)-b^2h_1h_2 \\
	&=-4a^2\sin \frac{h_1}{2} \sin \frac{h_2}{2} \cos \left(\frac{h_1+h_2}{2}\right)-b^2h_1h_2. \\
\end{align*}

Thus, we have
\begin{align*}
	4\text{Area}(\triangle A_0 A_1 A_2)^2 &= \left(4a^2\sin^2\frac{h_1}{2}+b^2h_1^2\right)\left(4a^2\sin^2\frac{h_2}{2}+b^2h_2^2\right)\\
	&-\left(4a^2\sin \frac{h_1}{2} \sin \frac{h_2}{2} \cos \left(\frac{h_1+h_2}{2}\right)+b^2h_1h_2\right)^2 \\
	&=16a^4\sin^2\frac{h_1}{2}\sin^2\frac{h_2}{2}\sin^2\frac{h_1+h_2}{2}+4a^2b^2\left(h_1^2\sin^2\frac{h_2}{2}  
+h_2^2\sin^2\frac{h_1}{2}\right)  \\
	& -8a^2b^2h_1h_2\sin \frac{h_1}{2} \sin \frac{h_2}{2} \cos\frac{h_1+h_2}{2} \\
	&=16a^4\sin^2\frac{h_1}{2}\sin^2\frac{h_2}{2}\sin^2\frac{h_1+h_2}{2}+4a^2b^2\left(h_1\sin\frac{h_2}{2}  
-h_2\sin\frac{h_1}{2}\right)^2  \\
	& +8a^2b^2h_1h_2\sin \frac{h_1}{2} \sin \frac{h_2}{2}\left(1- \cos\frac{h_1+h_2}{2} \right)\\
	&=16a^4\sin^2\frac{h_1}{2}\sin^2\frac{h_2}{2}\sin^2\frac{h_1+h_2}{2}+a^2b^2h_1^2h_2^2\left(\frac{\sin \frac{h_2}{2}}{\frac{h_2}{2}} -\frac{\sin \frac{h_1}{2}}{\frac{h_1}{2}}\right)^2 \\
	& +16a^2b^2h_1h_2\sin \frac{h_1}{2} \sin \frac{h_2}{2}\sin^2\frac{h_1+h_2}{4}.\\
\end{align*}

We proved the lemma in the case of a helix ${\mathcal H}(a,b)$ with parametrization \ref{sth}. 

 Now, consider  a helix  ${\mathcal H'}$ in general position,  with curvature $\kappa >0 $ and torsion $\tau > 0$. 
Let $A_0$, $A_1$, and $A_2$ be three distinct points  on the helix ${\mathcal H'}$   whose arclengths from the origin of the helix are $s_0$, $s_1$, $s_2$ respectively, with 
$0 < s_0 < s_1 < s_2$.

Here we use the uniqueness part of the Fundamental Theorem of the local theory of curves (for example, see do Carmo \cite{doC} p.19). It states that if $\alpha, \alpha' : I \to \R$ are regular parametrized curves with natural parameter $s$, 
which have the same curvature $\kappa(s)$ and the same torsion $\tau (s)$ for all $s$, then $\alpha'$ differs from $\alpha$ by a rigid motion, that is there exist an orthogonal linear map $\rho $ of $\R ^3$, with positive determinant, and a vector $c$ such that 
$\alpha' = \rho \circ \alpha + c$. 

Since the helix ${\mathcal H'}$ and the helix ${\mathcal H(a,b)}$ have the same curvature $\kappa $ and the same torsion $\tau$, there exists a rigid motion which maps the helix ${\mathcal H'}$ onto the helix ${\mathcal H(a,b)}$; 
 and it maps the points $A_0,A_1,A_2$ to some points $A_0',A_1',A_2'$ respectively, which are   on ${\mathcal H(a,b)}$ and have corresponding values of the natural parameter $s_0,s_1,s_2$. 

Note that the rigid motions do not change the distance between two points, and preserve the area  of  a triangle. Therefore,  $|A_0A_1|=|A_0'A_1'|$, $|A_1A_2|=|A_1'A_2'|$, $|A_0A_2| = |A_0'A_2'|$, Area$(\triangle A_0A_1A_2) = $ Area 
$(\triangle A_0'A_1'A_2') $. Moreover, $a$,$b$, $h_1$, and $h_2$ are uniquely determined by $s_0$, $s_1$, $s_2$, $\kappa$ and $\tau$. Therefore, the lemma holds in the case of a  helix  ${\mathcal H'}$ in general position.

\end{proof}

It is easy to estimate the terms $T_1$ and $T_3$. To estimate $T_2$ we use the following lemma.
\begin{lemma} \label{T2}
Let $h_1$ and $h_2$ be real numbers, both in the interval $(0,2\pi)$. Then, 
$$\left(\frac{\sin \frac{h_2}{2}}{\frac{h_2}{2}} -\frac{\sin \frac{h_1}{2}}{\frac{h_1}{2}}\right)^2  \leq\frac{\max ( h_1, h_2)^2}{144}(h_1-h_2)^2.$$
\end{lemma}

\begin{proof}
Applying the mean value theorem to the function ${\displaystyle \frac{\sin x}{x}}$, we have that there exists $\zeta$ between $\frac{h_1}{2}$ and  $\frac{h_2}{2}$  such that 
$$\left(\frac{\sin \frac{h_2}{2}}{\frac{h_2}{2}} -\frac{\sin \frac{h_1}{2}}{\frac{h_1}{2}}\right)^2 = \left(\frac{h_1}{2}- \frac{h_2}{2}\right)^2\left(\frac{\zeta \cos \zeta -\sin\zeta}{\zeta^2}\right)^2.$$
By Lemma \ref{lem:2}, we get that ${\displaystyle 0 < \frac{\sin x - x\cos x}{x^2} < \frac{x}{3}}$ for every $x \in (0, \pi )$. Therefore,
$$
	\left(\frac{\sin \frac{h_2}{2}}{\frac{h_2}{2}} -\frac{\sin \frac{h_1}{2}}{\frac{h_1}{2}}\right)^2 < \left(\frac{h_1}{2}- \frac{h_2}{2}\right)^2\left(\frac{\zeta }{3}\right)^2 
	\leq \frac{\max ( h_1, h_2)^2}{144}(h_1-h_2)^2.
$$

\end{proof}

In the next two lemmas we prove the existence of the positive constants ${\mathcal D_L}$ and ${\mathcal A_L}$  mentioned in the introduction.

\begin{lemma} \label{3vect}
Let $v_1$, $v_2$, and $v_3$ be linearly independent vectors in ${\mathbb R}^3$. Then, there exists a constant $c>0$ (depending on $v_1$, $v_2$, and $v_3$) such that 
$||m_1v_1+m_2v_2+m_3v_3|| \geq c$ for any integers $m_1$, $m_2$, and $m_3$ with $(m_1,m_2,m_3) \neq (0,0,0)$. 
\end{lemma}

\begin{proof}
We have $||m_1v_1+m_2v_2+m_3v_3||^2 = (m_1v_1+m_2v_2+m_3v_3)\cdot (m_1v_1+m_2v_2+m_3v_3) ={\displaystyle  \sum_{i=1}^3\sum_{j=1}^3 (v_i \cdot v_j) m_im_j :=Q(m_1,m_2,m_3)}$.

The quadratic form $Q(m_1,m_2,m_3)$ is positive definite, since $v_1$, $v_2$, and $v_3$ are linearly independent. Denote by $M(Q)$ the matrix of the quadratic form   $Q(m_1,m_2,m_3)$. Since $M(Q)$ is symmetric and  $Q(m_1,m_2,m_3)$ is positive definite, the eigenvalues of $Q(M)$  are real positive numbers, say $0 < \lambda_1 \leq \lambda_2 \leq \lambda_3$.  Moreover, by the Spectral Theorem for real symmetric matrices, there exists an orthonormal basis of vectors $w_1, w_2,w_3$ for 
${\mathbb R}^3$, where $w_1$, $w_2$, and $w_3$ are eigenvectors corresponding to the eigenvalues $ \lambda_1$,  $\lambda_2$, and $\lambda_3$. This implies that the minimum of $Q(m_1,m_2,m_3)$ on the unit sphere is $\lambda_1$.
Therefore, the lemma holds with $c = \sqrt{\lambda_1}$.
 
\end{proof} 

\begin{lemma} \label{lat}
Consider a lattice $\mathcal{L}=\mathcal{L}(v_0,v_1,v_2) = \{ v_0+m v_1 + nv_2 + pv_3: m,n,p \in \Z\}$. There exist positive constants ${\mathcal D_L}$ and ${\mathcal A_L}$  (depending on ${\mathcal L}$) such that

\noindent
(i)  if $A_0$ and $A_1$ are any two distinct lattice points in 
${\mathcal L}$, then $|A_0A_1| \geq {\mathcal D_L}$;

\noindent
and

\noindent 
(ii)  if $A_0$, $A_1$, and $A_2$ are any three non-collinear lattice points in ${\mathcal L}$, then   Area$(\triangle A_0A_1A_2) \geq {\mathcal A_L}$.

\end{lemma}

\begin{proof}
If $A_0$ and $A_1$ are   two distinct lattice points in 
${\mathcal L}$, then $\overrightarrow{A_0A_1} = m' v_1 + n'v_2 + p'v_3$ with $m'$, $n'$, $p'$ integers such that $(m',n',p') \neq (0,0,0)$. Now, (i) follows from Lemma \ref{3vect}. 

Next, let $A_0$, $A_1$, and $A_2$ be three non-collinear lattice points in ${\mathcal L}$. We have that Area$(\triangle A_0A_1A_2) = ||\overrightarrow{A_0A_1}\times \overrightarrow{A_0A_2}||/2$. 

Since $A_0$, $A_1$, and $A_2$  are  non-collinear, then Area$(\triangle A_0A_1A_2)  \neq 0$. Therefore, $\overrightarrow{A_0A_1}\times \overrightarrow{A_0A_2} \neq \overrightarrow{0}$. 

Furthermore,  since $\overrightarrow{A_0A_1} = m' v_1 + n'v_2 + p'v_3$ and  $\overrightarrow{A_0A_2} = m'' v_1 + n''v_2 + p''v_3$ for some integers $m',m'',n',n'',p'$, and $p''$, then 
$$\overrightarrow{A_0A_1}\times \overrightarrow{A_0A_2} = q( v_1\times v_2) + r(v_2 \times v_3) + s(v_1 \times v_3),$$
with $q=m'n'' - m'n''$, $r=n'p''-n''p'$, and $s=m'p''-m''p'$. 

Since, $\overrightarrow{A_0A_1}\times \overrightarrow{A_0A_2} \neq \overrightarrow{0}$ we have $(q,r,s) \neq (0,0,0)$.

Moreover, the vectors $v_1\times v_2$,  $v_2\times v_3$, and $v_1\times v_3$ are linearly independent. 

Indeed, if $c_1( v_1\times v_2) + c_2(v_2 \times v_3)+c_3(v_1 \times v_3) = 0$, then after taking the dot product of the last equation with $v_3$, we get 
$c_1((v_1 \times v_2) \cdot v_3)=0$. However, the triple product $(v_1 \times v_2) \cdot v_3 \neq 0$, since the vectors $v_1$, $v_2$, and $v_3$ are linearly independent. 
We show similarly that $c_2=0$ and $c_3=0$. 

Thus, the  vectors $v_1\times v_2$,  $v_2\times v_3$, and $v_1\times v_3$ are linearly independent. 

Now (ii) follows from Lemma \ref{3vect}.
\end{proof}

\section{Proof of the main result}

First, we consider the case when the lattice points are {\it on} the helix. 

\begin{theorem} \label{on}
Let ${\mathcal H}$ be a helix with curvature $\kappa >0 $ and torsion $\tau > 0$. Let $\mathcal{L}=\mathcal{L}(v_0,v_1,v_2,v_3)$ be an affine lattice in ${\mathbb R}^3$, and 
let $A_0$, $A_1$, and $A_2$ be three distinct points in ${\mathcal L}$ which are also on the helix ${\mathcal H}$ and whose arclengths from the origin of the helix are $s_0$, $s_1$, $s_2$ respectively, with 
$s_0 < s_1 < s_2$. Then ${\displaystyle |A_0A_2| \geq \min \left (  \frac{\pi \tau}{\kappa^2 + \tau^2},  1.5{\mathcal A_L}^{\frac{1}{3}}\kappa^{-\frac{1}{3}}\right )} $ and  the length of the arc $\arc{A_0A_2}$ is at least
${\displaystyle \min \left (\frac{\pi}{\sqrt{\kappa^2 + \tau^2}}, 2.4 {\mathcal A_L}^{\frac{1}{3}} \kappa^{-\frac{1}{3}}\right )}$.
\end{theorem}

\begin{proof} 
Define
   ${\displaystyle a = \frac{\kappa}{\kappa^2 + \tau^2}}$  and    ${\displaystyle  b = \frac{\tau}{\kappa^2 + \tau^2}}$. Recall that 
   $s = \sqrt{a^2+b^2}t$. Let ${\displaystyle t_i = \frac{s_i}{\sqrt{a^2+b^2}}}$, for $i=0,1,2$. Then $0 < t_0 < t_1 < t_2$ and 
   $A_i =  \vec{r}(t_i)$ for $i=0,1,2$. Let $h_1=t_1-t_0$ and $h_2 = t_2 - t_1$. 
   
  We consider two cases. 

\vskip 5pt
 {\bf Case I:}  $h_1 + h_2 \geq \pi$. 
  
  Then the arclength of the arc $\arc{A_0A_2}$  is  $$s_2-s_0 = (h_1+h_2)\sqrt{a^2+b^2} \geq \pi \sqrt{a^2+b^2} = \frac{\pi}{\sqrt{\kappa^2 + \tau^2}}.$$
  
  Also, by Lemma \ref{sthelix} (i) $$|A_0A_2| \geq b(h_1+h_2) \geq \pi b =     \frac{\pi \tau}{\kappa^2 + \tau^2}.$$
  
  Note that one cannot expect to get a nontrivial lower bound on $|A_0A_2|$   depending only on $\kappa $. For example, consider the case when $a$ is a large integer,  $b = \frac{1}{2\pi }$, $t_0=0$, $t_1 = 2\pi$,  and $t_2 = 4\pi $. 

\vskip 5pt
 {\bf  Case II:}  $h_1 + h_2 < \pi$.

   Denote by  $S$   the area of $\triangle A_0A_1A_2$. From Lemma  \ref{sthelix} (ii) we have $4S^2 = T_1 + T_2+ T_3$, where 
   
   $T_1 =16a^4\sin^2\frac{h_1}{2}\sin^2\frac{h_2}{2}\sin^2\frac{h_1+h_2}{2}$, 

$T_2= a^2b^2h_1^2h_2^2\left(\frac{\sin \frac{h_2}{2}}{\frac{h_2}{2}} -\frac{\sin \frac{h_1}{2}}{\frac{h_1}{2}}\right)^2$, and

$ T_3 =  16a^2b^2h_1h_2\sin \frac{h_1}{2} \sin \frac{h_2}{2}\sin^2\frac{h_1+h_2}{4}$. 
   
  Since $0 < h_1 + h_2 < \pi$, we have $T_1 > 0$, $T_2 \geq 0$, and $T_3 >0$. Therefore, $S > 0$, so the points $A_0$, $A_1$, and $A_2$ are non-collinear. Now, 
  by Lemma \ref{lat} we have $S \geq {\mathcal A_L} > 0$.  We obtain 
\begin{equation} \label{ATb}
4{\mathcal A_L}^2 \leq T_1 + T_2 + T_3. 
\end{equation}

Next we get bounds for $T_1$, $T_2$, and $T_3$. 

Since $0 < h_1 + h_2$ and $\sin x < x$  when $x > 0$ we obtain 
\begin{equation} \label{T1b}
   T_1 \leq \frac{a^4 h_1^2 h^2 (h_1+h_2)^2}{4} \leq \frac{a^4(h_1+h_2)^6}{64}. 
\end{equation}
(Here and below, we use the inequality $ab  \leq (a+b)^2/4$.)

Next, by Lemma \ref{T2} we have
$$2{\mathcal A_L}^2 \leq  T_2 \leq a^2b^2 h_1^2h_2^2  \frac{\max ( h_1, h_2)^2}{144}(h_1-h_2)^2 < a^2b^2  \frac{(h_1+h_2)^8}{2304}.$$
 
Since, $h_1 + h_2 < \pi$, we obtain, 
\begin{equation} \label{T2b}
  T_2 < a^2b^2 \frac{\pi ^2 (h_1+h_2)^6}{2304}.
\end{equation}

Again,  using $six < x$ for $x>0$ again, we obtain
 \begin{equation} \label{T3b}
  T_3 < a^2b^2 h_1^2 h_2^2 \frac{(h_1+h_2)^2}{4} \leq a^2b^2 \frac{(h_1+h_2)^6}{64}.
  \end{equation}
  
  Combining equation \eqref{ATb} with the bounds \eqref{T1b}, \eqref{T2b}, and \eqref{T3b}, we obtain
$$4{\mathcal A_L}^2 \leq  (h_1+h_2)^6a^2 \left ( \frac{a^2}{64} + b^2 \left ( \frac{1}{64} + \frac{\pi ^2}{2304} \right )  \right ) .$$  

 Now $  \frac{1}{64} + \frac{\pi ^2}{2304} < 0.02$. Therefore, 
 
 \begin{equation} \label{oncl}
 200{\mathcal A_L}^2  < (h_1+h_2)^6a^2 (a^2+b^2).
\end{equation}

Also, $200^{1/6} > 2.4$, so
\begin{equation}\label{h1h2c2}
h_1 + h_2   > 2.4 {\mathcal A_L}^{\frac{1}{3}}a^{-\frac{1}{3}}(a^2+b^2)^{-\frac{1}{6}},
\end{equation}
in Case 2.

Recall that $s_2 - s_0 =(h_1+h_2)\sqrt{a^2 + b^2}$,    ${\displaystyle a = \frac{\kappa}{\kappa^2 + \tau^2}}$,  and    ${\displaystyle  a^2+b^2 = \frac{1}{\kappa^2 + \tau^2}}$.

We get, 
\begin{equation} \label{s0s2c2}
s_2 - s_0  >  2.4 {\mathcal A_L}^{\frac{1}{3}} \kappa^{-\frac{1}{3}},
\end{equation}
completing the stated lower bound for the arclength of $\arc{A_0A_2}$.

Next, we obtain a lower bound for $|A_0A_2|$.  
 
 In this case, using Lemma \ref{lem:2} we obtain 
 $  \sin \left ( \frac {h_1 + h_2}{2}\right )  \geq \frac{1}{\pi }(h_1+h_2) $. By Lemma \ref{sthelix}  (i) 
 $|A_0A_2|^2=4a^2\sin^2\left(\frac{h_1+h_2}{2}\right)+b^2(h_1+h_2)^2$. Therefore,
 
       $$|A_0A_2| \geq \frac{2}{\pi } (h_1+h_2)(a^2+b^2)^{\frac{1}{2}} =\frac{2}{\pi} (s_2-s_0), $$ in Case 2.
 Using \eqref{s0s2c2} we get ${\displaystyle |A_0A_2| \geq \frac{4.8}{\pi}{\mathcal A_L}^{\frac{1}{3}} \kappa^{-\frac{1}{3}}}$.
 
 Since $\frac{4.8}{\pi} > 1.5$, we obtain the stated bound for $|A_0A_2|$. 
 
 We have established the theorem in the case of a helix ${\mathcal H}(a,b)$ with parametrization $\vec{r}(t)=\langle a\cos t, a \sin t, bt \rangle$.

\end{proof}

 Next we prove Corollary \ref{cor:1}. 
 
 \begin{proof} 
 First, by Lemma \ref{lem:1} in the case of the standard lattice in $\R^3$, ${\mathcal A_L} = \frac{1}{2}$. Therefore, the condition $\tau \kappa^{\frac{1}{3}} \geq 0.4(\kappa^2 + \tau ^2)$ implies 
 ${\displaystyle    \frac{\pi \tau}{\kappa^2 + \tau^2} > 1.5{\mathcal A_L}^{\frac{1}{3}}\kappa^{-\frac{1}{3}}} $ in the case of standard lattice  ${\mathcal L}$. The corollary follows by 
 noting that $\frac{1.5}{\sqrt[3]{2}} > 1.1$.
 \end{proof}
.

Finally, we prove Theorem \ref{main}.

\begin{proof}

Let $$0 < \delta \leq \min \left (\frac{\mathcal D_L}{4},  \frac{{\mathcal D_L}^2}{11\pi^3}(a^2+b^2)^{-1},  \frac{2{\mathcal A_L}(a^2+b^2)^{-\frac{1}{2}}}{11\pi } \right ).$$    Let  $A_0'$, $A_1'$, and $A_2'$ be three distinct lattice points 
within $\delta $ of the helix ${\mathcal H}$. Therefore, there exist points $A_0$, $A_1$, and $A_2$ on the helix, such that 
$$|A_0A_0'| \leq \delta, \quad |A_1A_1'| \leq \delta, \quad \mbox{ and } |A_2A_2'| \leq \delta.$$

Since $A_0'$, $A_1'$, and $A_2'$  are in ${\mathcal L}$,   by Lemma \ref{lat} we have $|A_0'A_1'| \geq {\mathcal D_L}$, $|A_0'A_2'| \geq {\mathcal D_L}$, and $|A_1'A_2'| \geq {\mathcal D_L}$. 
Since $\delta \leq {\mathcal D_L}/4$, by the triangle inequality,  we get 
$$|A_0 A_2| \geq   |A_0'A_2'| - |A_0A_0'| - |A_2A_2'| \geq {\mathcal D_L} - 2\delta \geq {\mathcal D_L}/2.$$ Similarly, $|A_0A_1| \geq {\mathcal D_L}/2$ and 
$|A_1A_2| \geq {\mathcal D_L}/2$
Thus, the points $A_0$, $A_1$, and $A_2$ are distinct. Let the values of the natural parameter   corresponding to $A_0$, $A_1$, and $A_2$ be $s_0$, $s_1$, and $s_2$, respectively. 
Without loss of generality we can assume $s_0 < s_1 < s_2$ (otherwise we relabel $A_0'$, $A_1'$, and $A_2'$).  As before, let $t_i = s_i (a^2+b^2)^{-\frac{1}{2}}$
for $i=0,1,2$,  $h_1 = t_1 - t_0 > 0$ and $h_2 = t_2 - t_1 >0$. 

We consider two cases.

\vskip 5pt
Case I: $h_1 + h_2 \geq \pi$. 

In this case, in exactly  the same way as in the proof of Theorem \ref{on},  we obtain $$ |A_0A_2| \geq     \frac{\pi \tau}{\kappa^2 + \tau^2}.$$
Now,  by the triangle inequality,  we get $$|A_0' A_2'| \geq |A_0A_2| - |A_0A_0'| - |A_2A_2'| \geq |A_0A_2| -2\delta  ,$$
so  $$ |A_0'A_2'| \geq \frac{\pi \tau}{\kappa^2 + \tau^2} - 2\delta .$$

\vskip 5 pt
Case II. $h_1 + h_2 \leq \pi$. 

Here, first we show that $A_0'$, $A_1'$ and $A_2'$ are not on a straight line.

 Recall that $|A_0A_1| \geq {\mathcal D_L}/2$. Therefore, $s_1 - s_0  \geq |A_0A_1| \geq   {\mathcal D_L}/2$. Since $s_1-s_0 = h_1 (a^2+b^2)^{\frac{1}{2}}$, we get $h_1 \geq \frac{ {\mathcal D_L}}{2} (a^2+b^2)^{-\frac{1}{2}}$. 
 Similarly, $s_2 - s_1 \geq  {\mathcal D_L}/2$ and 
 $h_2 \geq \frac {\mathcal D_L}{2} (a^2+b^2)^{-\frac{1}{2}}$.

Denote the area of $\triangle A_0 A_1 A_2$ by $S$.

By Lemma \ref{sthelix} (ii),  we have  $4S^2 = T_1 + T_2 + T_3$, where $T_1 = 16a^4\sin^2\frac{h_1}{2}\sin^2\frac{h_2}{2}\sin^2\frac{h_1+h_2}{2}$, $T_2 \geq 0$, and $T_3 >0$. 

  Since $h_1 + h_2 \leq \pi$, by Lemma 
\ref{lem:2} $$T_1 \geq \frac{16}{\pi^6} h_1^2 h_2^2 (h_1 + h_2)^2.$$
Therefore, 
${\displaystyle   4S^2 \geq   \frac{16}{\pi^6} h_1^2 h_2^2 (h_1 + h_2)^2}$. 
We get, 
\begin{equation} \label{sl lb}
S  \geq \frac{2}{\pi^3}h_1h_2(h_1+h_2). 
\end{equation}

 Denote the area of $\triangle A_0' A_1' A_2'$ by $S_3$.

 Assume that the  points $A_0'$, $A_1'$ and $A_2'$ {\it are} on a straight line. Then, 
 $S_3  = 0$.

 Applying Lemma \ref{lem:4} to $\triangle A_0 A_1 A_2$ and $\triangle A_0' A_1' A_2'$ we obtain
 \begin{equation} \label{ar dif}
 \left | S  -  S_3 \right |  \leq 
 \frac{\delta (|A_0A_1| + |A_1A_2| + |A_0A_2|)}{2} + \frac{3\delta ^2}{2}.
 \end{equation}
 Now, $|A_0A_1| < s_1 - s_0$, $|A_1A_2| < s_2 - s_1$, and $|A_0A_2| < s_2 - s_0$, so we obtain
 $$S \leq \delta (s_2-s_0) + \frac{3\delta ^2}{2}.$$
 Since, $s_2 - s_0 \geq |A_0A_1| + |A_1A_2| \geq {\mathcal D_L}$ and $\delta \leq {\mathcal D_L}/4$, we have $\delta \leq (s_2-s_0)/4$, so
\begin{equation} \label{sup}
S \leq \frac{11}{8} \delta (s_2-s_0)=\frac{11}{8}\delta (h_1+h_2)(a^2+b^2)^{-\frac{1}{2}} .
\end{equation}

Combining \eqref{sup} and   \eqref{sl lb} we obtain
$$\frac{2}{\pi^3}h_1h_2(h_1+h_2) \leq    S \leq \frac{11}{8}\delta (h_1+h_2)(a^2+b^2)^{-\frac{1}{2}} .$$

Therefore,
$$\frac{4}{11\pi^3}h_1h_2 \leq \delta .$$

Now, $h_1 = t_1 - t_0 = (s_1 - s_0)(a^2+b^2)^{-\frac{1}{2}} > |A_0A_1|(a^2+b^2)^{-\frac{1}{2}} \geq \frac{\mathcal D_L}{2}(a^2+b^2)^{-\frac{1}{2}}$. Similarly, 
$h_2 >  \frac{\mathcal D_L}{2}(a^2+b^2)^{-\frac{1}{2}}$. Therefore,

$$\delta > \frac{{\mathcal D_L}^2}{11\pi^3}(a^2+b^2)^{-1},$$
which contradicts the upper bound for $\delta $, so our assumption is false. Hence, the points 
$A_0'$, $A_1'$ and $A_2'$ {\it are not} on a straight line. By Lemma \ref{lat} $S_3$, the area of 
$\triangle A_0' A_1' A_2'$ is at least ${\mathcal A_L}$.

Next,  we get a lower bound for the area of $\triangle A_0A_1A_2$. Using \eqref{ar dif} we obtain 
$$ S \geq {\mathcal A_L}  -  \frac{\delta (|A_0A_1| + |A_1A_2| + |A_0A_2|)}{2} - \frac{3\delta ^2}{2}.$$

Now, $|A_0A_1| + |A_1A_2| + |A_0A_2| < 2(s_2 - s_0)$ and we showed above that $\delta \leq (s_2-s_0)/4$.
Therefore,
$$S \geq {\mathcal A_L}  - \frac{11\delta (s_2-s_0)}{8}.$$

We have
 $s_2 - s_0 = 2(h_1+h_2)(a^2 + b^2)^{\frac{1}{2}} \leq 2\pi(a^2 + b^2)^{\frac{1}{2}} $. 

We obtain, 
$$S \geq {\mathcal A_L}  - \frac{11\delta \pi(a^2 + b^2)^{\frac{1}{2}}}{4}.$$

Now, by the upper bound on $\delta$, we have 
$$\frac{11\delta \pi(a^2 + b^2)^{\frac{1}{2}}}{4} < \frac{\mathcal A_L}{2}.$$

Therefore, 
\begin{equation}   \label{s3lb}
S > \frac{\mathcal A_L}{2}.
\end{equation}

Now, we get lower bound for $h_1 + h_2$ the same way we did in the case of lattice points {\it on the helix}, in Case 2 of the proof of Theorem \ref{on}. The upper bounds of $T_1$, $T_2$, and $T_3$ are the same as before, and the estimates 
\eqref{T1b}, \eqref{T2b}, and \eqref{T3b} still hold. 

The only difference is that the lower bound for the area of $\triangle A_0 A_1 A_2$ now is ${\mathcal A_L}/2$ rather than ${\mathcal A_L}$.  

 The analogue equation \eqref{oncl} is 
 \begin{equation} \label{cl}
 50{\mathcal A_L}^2  < (h_1+h_2)^6a^2 (a^2+b^2),
 \end{equation}
 (with the difference that now the left-hand-side of \eqref{cl}  is $1/4$ of the left-hand-side of \eqref{oncl}).
 
 Also, $50^{1/6} > 2.4$, so
\begin{equation}\label{h1h2c2n}
h_1 + h_2   > 1.9 {\mathcal A_L}^{\frac{1}{3}}a^{-\frac{1}{3}}(a^2+b^2)^{-\frac{1}{6}}.
\end{equation}

Using Lemma \ref{lem:2} we obtain 
 $  \sin \left ( \frac {h_1 + h_2}{2}\right )  \geq \frac{1}{\pi }(h_1+h_2) $. By Lemma \ref{sthelix}  (i) 
 $|A_0A_2|^2=4a^2\sin^2\left(\frac{h_1+h_2}{2}\right)+b^2(h_1+h_2)^2$. Therefore,
 
       $$|A_0A_2| \geq \frac{2}{\pi } (h_1+h_2)(a^2+b^2)^{\frac{1}{2}} . $$

 Using \eqref{h1h2c2n} and  ${\displaystyle  \kappa = \frac{a}{a^2+b^2}}$ we get $$ |A_0A_2| \geq \frac{3.8}{\pi}{\mathcal A_L}^{\frac{1}{3}} \kappa^{-\frac{1}{3}}.$$
 
 Since $\frac{3.8}{\pi} > 1.2$,  and $|A_0'A_2'| \geq |A_0A_2| - 2\delta$ we obtain the stated bound for $|A_0'A_2'|$.

\end{proof}

\end{document}